\documentclass[a4paper]{amsart}
\usepackage{tikz}
\usepackage{tikz-cd}
\usetikzlibrary{decorations.markings,decorations.pathmorphing,cd}
\usetikzlibrary{arrows}
\usetikzlibrary{calc}
\usepackage{pgfplots}
\pgfplotsset{every axis/.append style={
                    axis x line=middle,    
                    axis y line=middle,    
                    axis line style={-,color=blue}, 
                    xlabel={$x$},          
                    ylabel={$y$},          
            }}
            
\usepackage[latin1]{inputenc}
\usepackage{amsmath}
\usepackage{amsthm}
\usepackage{amsfonts}
\usepackage{amssymb}
\usepackage{enumerate}
\usepackage{enumitem}
\usepackage{graphicx}
\usepackage{hyperref}
\usepackage{nicefrac}
\usepackage{cancel}

\def\ZZ{\mathbb{Z}}

\def\CC{\mathbb{C}}

\def\FF{\mathbb{F}}

\def\SS{\mathbb{S}}
\def\PP{\mathbb{P}}

\newtheorem{thm}{Theorem}[section]  
\newtheorem{main-thm}{Theorem}  
\newtheorem{main-conj}{Conjecture}[section]%
\newtheorem{cor}{Corollary}[section]
\newtheorem{lemma}{Lemma}[section]

\theoremstyle{remark}
\newtheorem{rem}{Remark}[section]

\theoremstyle{definition}
\newtheorem{dfn}{Definition}[section]
\newtheorem{exam}{Example}[section]

\makeatletter
\let\c@lemma\c@thm
\let\c@prop\c@thm
\let\c@propdef\c@thm
\let\c@proper\c@thm
\let\c@problem\c@thm
\let\c@conj\c@thm
\let\c@cor\c@thm
\let\c@rem\c@thm
\let\c@dfn\c@thm
\let\c@notation\c@thm
\let\c@exam\c@thm
\makeatother
\title[Homotopy type of complements of fiber-type curves]{Homotopy type of complements of fiber-type curves}

\author[Jos\'e I. Cogolludo-Agust{\'i}n and Eva Elduque]{J.I.~Cogolludo-Agust{\'i}n and Eva Elduque}
\address{Departamento de Matem\'aticas, IUMA\\
Universidad de Zaragoza \\
C.~Pedro Cerbuna 12 \\
50009 Zaragoza, Spain.} 
\email{jicogo@unizar.es} 

\address{Departamento de Matem\'aticas, ICMAT\\ 
Universidad Aut\'onoma de Madrid \\
28049 Madrid, Spain.}
\email{eva.elduque@uam.es}

\makeindex

\begin{document}

\thanks{The authors are partially supported by the Spanish Government PID2020-114750GB-C31. 
The first author is partially supported by the Departamento de Ciencia, Universidad y Sociedad del 
Conocimiento del Gobierno de Arag{\'o}n (Grupo de referencia E22\_20R ``{\'A}lgebra y Geometr{\'i}a''). The second author is partially supported by the Ram\'on y Cajal Grant RYC2021-031526-I funded by MCIN/AEI /10.13039/501100011033 and by the European Union NextGenerationEU/PRTR} 


\begin{abstract}
In this note we study the homotopy type of the complement of a plane projective curve of fiber-type.
Roughly speaking, a curve of fiber-type is a finite union of fibers of a pencil.
Under some restrictions, a full description of their homotopy type is given.
In the context of fiber-type curves, we also partially answer in the positive a question originally posed by Libgober
on the homotopy type of complements of plane curves, namely, whether or not the homotopy type of a curve complement 
is determined by the fundamental group and its Euler characteristic.
\end{abstract}

\maketitle

\begin{quotation}
\emph{
With gratitude we celebrate Alejandro Melle's 55th and Enrique Artal's 60th birthday.
They have both been and still are a source of inspiration and a joy to work with.
}
\end{quotation}

\section{Introduction}
In a classical paper from the 70's, M.J.~Dunwoody~\cite{Dunwoody-homotopy} exhibited two 2-complexes whose fundamental groups are
isomorphic (the group of the trefoil knot), whose Euler characteristics are both one, but whose homotopy types are not equivalent.
A decade later, A.~Libgober~\cite{Libgober-homotopytype} described the homotopy type of the complement of affine plane curves
(transversal at infinity) in terms of their braid monodromy. In this celebrated paper, Libgober posed the question about whether
or not Dunwoody's example could be realized as complements of algebraic curves. We refer to this as \emph{Libgober's Question}
(on Homotopy Type)~\cite[p.114]{Libgober-homotopytype}, namely,
\begin{quotation}
\label{quote:libgober}
\emph{
Do two curves in $\CC^2$ exist whose complements have the same fundamental group and Euler characteristic, but are not
homotopy equivalent?
}
\end{quotation}
Or more generally:
\begin{quotation}
\emph{
Is the homotopy type of the complement of a projective plane curve determined by its Euler characteristic and fundamental group?
}
\end{quotation}
Consider a curve $C\subset\PP^2$, a line $\ell\not\subset C$ and a point $P\in\ell\setminus C$. By the continuity of the roots,
the pencil $\mathcal L$ of lines passing through $P$ other than $\ell$, and away from the lines intersecting $C$ non-transversally, produces a family of braids obtained by the intersection
of the lines in $\mathcal L$ with $C$. The braid monodromy of a curve contains information that goes beyond their homotopy type,
in fact it determines its embedded topology. This was first shown for curves with nodes and cusps by Kulikov-Teicher
in~\cite{Kulikov-Teicher-braid} and later for general curves by Carmona in~\cite{Carmona-tesis}.
In this context, Artal and several authors have contributed to the study of braid monodromy, topology, and homotopy type in a
series of papers~\cite{ACO-kummer,Artal-Carmona-ji-braid,Artal-Carmona-ji-effective}.
In this note we approach the homotopy type problem from a different perspective, namely for fiber-type curves, as defined below.

Let $X$ be a smooth complex projective surface. We say that a
curve $C$ in $X$ is a \emph{fiber-type curve} in $X$ if $C$ is the closure in $X$ of a finite number of fibers of a surjective
algebraic morphism with connected generic fibers $F:U\to S$, where $U$ is the complement of a finite number of points in $X$,
and $S$ is a smooth algebraic curve. Throughout this paper, we will always assume that $U$ is the maximal domain of definition
of $F:X\dashrightarrow S$.

The topology of the complements in $X$ of fiber-type curves was explored by the authors in recent
papers~\cite{ji-Eva-Generic,ji-Eva-GOMP,ji-Eva-orbifold,ji-Eva-Zariski}.
By~\cite[Theorem 1.3]{ji-Eva-orbifold}, if a curve $C$ in $X$ satisfies that $\pi_1(X\setminus C)$ is a free product of
cyclic groups, then $C$ is a fiber-type curve. However, the converse is false: if $C=\overline{F^{-1}(B)}$ for some
finite non-empty set $B\subset S$, so that $F:X\setminus C\to S\setminus B$ is a locally trivial fibration
(in other words, if $B\supset B_F$ where $B_F$ is the set of atypical values of $F$), then $\pi_1(X\setminus C)$ is rarely
a free product of cyclic groups~\cite[Lemma 2.18, Corollary 2.19]{ji-Eva-orbifold}. On the other hand, under mild assumptions
on $F$, and if $B\cap B_F=\emptyset$, then $\pi_1\left(\PP^2\setminus \overline{F^{-1}(B)}\right)$ is a free product of
cyclic groups by \cite[Theorem 1.2]{ji-Eva-orbifold}. In this sense, the situations $B\cap B_F=\emptyset$ and $B_F\subset B$
can be seen as extreme cases.

Moreover, the examples of~\cite{ji-Eva-Zariski} show that the topology of the complement of a fiber-type curve is not determined by
local invariants of the singularities of the curve, namely, there exist infinitely many pairs of curves in $\PP^2$ with the same
combinatorics and homeomorphic tubular neighborhoods such that the fundamental groups of their complements are not isomorphic,
one of them being a free product of cyclic groups, unlike the other one.

The first purpose of this note is to partially answer the projective version of Libgober's question in the positive for the family of groups that are isomorphic to a finite free product of cyclic groups. In particular, in section~\ref{sec:libgober} we prove it for free and
finite cyclic groups. We also show a partial answer for finite free products of cyclic groups, namely if $G$ is such a product, then,
for any projective plane curve whose fundamental group is isomorphic to $G$, the Euler characteristic determines all homotopy groups of its complement.

Note that the family of fiber-type curves contains cases of curves not previously considered by Libgober in~\cite{Libgober-homotopytype},
for instance, curves in $\PP^2$ or curves in $\CC^2\equiv\PP^2\setminus L_\infty$ that are non-transversal with respect to the line at infinity~$L_\infty$ (see Example~\ref{ex:pq}).

The second purpose is to give an alternative description for the homotopy type of the complement of a fiber-type curve $C$ in $\PP^2$ which is given as a finite union of fibers of a morphism $F$. More precisely, we do this in section~\ref{sec:homotopy} in the case when
$F_{|\PP^2\setminus C}$ is a fibration, or equivalently, in the case when $C=\overline{F^{-1}(B)}$ for a non-empty finite set $B$ such that $B_F\subset B$. The description of the homotopy type is given as the $2$-complex associated to a certain presentation of the fundamental group of the complement obtained by using the fibration structure.


%

\section{Preliminaries}

\begin{dfn}
	A continuous map $p:X\to Y$ between topological spaces is a \emph{quasifibration} if the induced morphisms
	$$
	\pi_i(X,p^{-1}(y),x)\to\pi_i(Y,y)
	$$
	are isomorphisms for all $i\geq 1$,  $y\in Y$ and all $x\in p^{-1}(y)$ (where the word ``isomorphism'' is interpreted as ``bijection'' if $i=1$, the case in which the domain is not necessarily a group).
\end{dfn}

By \cite[Theorem 4.41]{hatcher}, fibrations are examples of quasifibrations, and the long exact sequence of homotopy groups associated to a fibration with path-connected base (coming from the long exact sequence of a pair) also holds for a quasifibration. This long exact sequence is functorial with respect to fiberwise maps $X\to X'$ of quasifibrations over the same base $Y$.

\begin{dfn}
	For every topological space $X$ and every continuous function $f:X\to X$, we denote by $T_f$ its \emph{mapping torus}, that is
	$$
	T_{f}=X\times [0,1]/\left((x,1)\sim (f(x),0)\quad\text{ for all } x\in X\right).
	$$
\end{dfn}

Associated to each mapping torus, we have a projection morphism
$T_f\to[0,1]/(0\sim 1)\cong\SS^{1}$ given by $[(x,t)]\mapsto [t]$.

The following result is well known, but we include the proof for the reader's convenience.
\begin{lemma}\label{lem:qf}
	Let $K$ be the space resulting from gluing the mapping tori of $f_j:X\to X$ along $X\times\{0\}$ for all $j=1,\ldots, n$. If all the $f_j$'s are homotopy equivalences, then the projection map
	$$
	F:K\to \bigvee_{j=1}^n\SS^1_j
	$$
	obtained by gluing the projection maps $T_{f_j}\to\SS^1_j$ (where $\SS^1_j$ is a copy of $\SS^1$) is a quasifibration.
\end{lemma}
\begin{proof}
	Let $U$ be a small open ball centered at the wedge point $P$ of $\bigvee_{j=1}^n\SS^1_j$. Let $V_j=\SS^1_j\cup U$. By \cite[Lemma 4K.3 (a)]{hatcher}, $F$ is a quasifibration if
	$$
	F:F^{-1}(V_j)\to V_j, \quad\text{and}\quad F:F^{-1}(U)\to U
	$$
	are quasifibrations. Now, \cite[Lemma 4K.3 (c)]{hatcher} implies the following:
	\begin{itemize}
		\item  $F:F^{-1}(U)\to U$ is a quasifibration if $F:X\to \{P\}$ is a quasifibration and $f_j:X\to X$ is a weak homotopy equivalence for all $j=1,\ldots,n$. Note that $F:X\to \{P\}$ is a quasifibration (since it is a fibration) and that $f_j$ is a homotopy equivalence by hypothesis. Hence, $F:F^{-1}(U)\to U$ is a quasifibration.
		\item $F:F^{-1}(V_j)\to V_j$ is a quasifibration if $F:F^{-1}(\SS^1_j)\to \SS^1_j$ (the mapping torus projection $T_{f_j}\to \SS^1_j$) is a quasifibration and $f_k:X\to X$ is a weak homotopy equivalence for all $k\neq j$.
	\end{itemize}
	Hence, we just need to show that the mapping torus projection $F:T_{f_j}\to \SS^1_j$ is a quasifibration for all $j=1,\ldots,n$. Let $P_j\in \SS^1_j\setminus\{P\}$. Note that $T_{f_j}\to \SS^1_j$ is a fibration when restricted to $\SS^1_j\setminus\{P\}$ and to $\SS^1_j\setminus\{P,P_j\}$. Also, since the constant map $F:X\to\{P\}$ is a fibration and $f_j$ is a homotopy equivalence, \cite[Lemma 4K.3 (c)]{hatcher} implies that $F:T_{f_j}\setminus F^{-1}(P_j)\to \SS^1_j\setminus\{P_j\}$ is a quasifibration. Now, \cite[Lemma 4K.3 (a)]{hatcher} implies that $T_{f_j}\to \SS^1_j$ is a quasifibration. This concludes our proof.
\end{proof}

\begin{cor}\label{cor:aspherical}
	Under the hypotheses of Lemma~\ref{lem:qf}, if $X$ is aspherical, then so is $K$.
\end{cor}
\begin{proof}
	This follows from Lemma~\ref{lem:qf} by using the long exact sequence of the quasifibration in homotopy groups, and the fact that a wedge of $\SS^1$'s is aspherical.
\end{proof}

\section{A question of Libgober}\label{sec:libgober}
In~\cite{Libgober-homotopytype}, Libgober asked whether the 
fundamental group and the Euler characteristic of a curve in $\CC^2$ determines its homotopy type.

The following result shows that if $D$ is a curve in $\PP^2$ such that $\pi_1(\PP^2\setminus D)$ is free, 
then the homotopy type of $\PP^2\setminus D$ is completely determined by $\pi_1(\PP^2\setminus D)$ and its 
Euler characteristic. 

\begin{thm}
Let $D$ be a plane curve in $\PP^2$ whose fundamental group of the complement $\pi_1(\PP^2\setminus D)$ is a 
free group of rank $r$, then the homotopy type of $\PP^2\setminus D$ depends only on $r$ and $\chi(D)$ the 
Euler characteristic of $D$. Moreover,
\begin{equation}\label{eq:homtype}
\PP^2\setminus D\cong \left(\bigvee_r \SS^1\right) \bigvee \left(\bigvee_{s} \SS^2\right),
\end{equation}
where $s=\chi(\PP^2\setminus D)+r-1=2+r-\chi(D)$ and $\bigvee_{n} X$ denotes the wedge sum of $n$ copies of $X$.
\end{thm}
\begin{proof}
Since the space $\PP^2\setminus D$ is affine and smooth, $\PP^2\setminus D$ has the homotopy type of a 
2-dimensional CW complex (\cite{Andreotti-Frankel-Lefschetz} and \cite{Milnor1}).
In this case, the hypotheses of~\cite[Proposition 3.3.]{wall-FinitenessConditions} are satisfied and 
hence~\eqref{eq:homtype} follows. The value of $s$ can be obtained using $3-\chi(D)=\chi(\PP^2\setminus D)=1-r+s$.
\end{proof}

Note that the previous result includes complements of curves in $\PP^2$ that cannot be described as
complements of curves in $\CC^2$ (provided that their fundamental group is free),
as shown in the following example of \cite{ji-Eva-orbifold}, which we recall below.
 
 \begin{exam}[Example 4.12 in \cite{ji-Eva-orbifold}]
 \label{ex:pq}
 	Let $p,q\in \ZZ_{>0}$ be such that $p<q$ and $\gcd(p,q)=1$, and let $f_p,f_q\in \CC[x,y,z]$ be irreducible polynomials
 	of degrees $p$ and $q$. Suppose that $\pi_1(\PP^2\setminus V(f_p\cdot f_q))\cong \ZZ$ (which is the case if
 	$V(f_p)$ and $V(f_q)$ are transversal curves which only have nodal singularities, for example). Assume $F:=[f_p^q:f_q^p]$ has no multiple fibers outside $B_0:=\{[0:1], [1:0]\}$.
 	Let $B=B_0\cup B'$, where $B'\subset \PP^1\setminus (B_0\cup B_F)$ is a finite set of $r-1$ elements, $r\geq 1$. For $C:=\overline{F^{-1}(B)}$ one has
 	$$
 	\pi_1(\PP^2\setminus C)\cong \FF_r,
 	$$
 	where $\FF_r$ is the free group of rank~$r$.
 	If $p=1$, then one can see $C$ as a union $C=L\cup D$ where $L=\overline{F^{-1}([0:1])}$
 	is a line. Identifying $\CC^2\equiv\PP^2\setminus L$ one has
 	$\PP^2\setminus C=\CC^2\setminus D$, where $D$ and $L$ are not necessarily transversal.
 	On the other hand, if $p>1$, then $C$ does not contain a line, so $\PP^2\setminus C$ cannot be interpreted as the complement
 	of a curve in~$\CC^2$.
 \end{exam}

\begin{rem}
An analogous result can be proved for $\pi_1(\PP^2\setminus D)\cong \ZZ_d$ since the homotopy type is determined
by $d$ and $\chi(\PP^2\setminus D)$ as 
$$\PP^2\setminus D\cong (\SS^1 \cup_d e^2) \bigvee \left(\bigvee_{s} \SS^2\right),$$
where the gluing in $\SS^1 \cup_d e^2$ is such that the boundary of the $2$-cell $e^2$ is attached to $\SS^1$ by a map of order $d$, and $s=\chi(\PP^2\setminus D)-1=2-\chi(D)$ (see~\cite{dyer-sieradski,metzler}).
\end{rem}

The general question about (finitely generated) free products of cyclic groups groups is open, see~\cite{Hog-homotopytype} for this question. Here, we provide the following intermediate result.

\begin{thm}\label{thm:homotopyorbi}
Let $D$ be a plane curve in $\PP^2$ whose fundamental group of the complement $\pi_1(\PP^2\setminus D)$ is a
free product of cyclic groups. Then, the homotopy groups $\pi_i(\PP^2\setminus D)$ are determined for all
$i\geq 1$ by $\pi_1(\PP^2\setminus D)$ and $\chi(D)$.
\end{thm}

\begin{proof}
	By \cite[Remark 3.5]{ji-Eva-orbifold}, $\pi_1(\PP^2\setminus D)\cong \FF_r*\ZZ_p*\ZZ_q$, where $r\geq 0$, $p,q\geq 1$ 
	and $\gcd(p,q)=1$ ($\FF_0$ is the trivial group by convention). Now, let $\pi:\pi_1(\PP^2\setminus D)\to \ZZ_{pq}$ 
	be the composition of the abelianization and the projection onto the torsion, which is a surjective group homomorphism.
	Using Reidemeister-Schreier techniques, one can show that the kernel of $\pi$ is a free group of rank $pqr+(p-1)(q-1)$.
	Hence, there exists an unramified cover of degree $pq$ of $\PP^2\setminus D$ with free fundamental group, which, like 
	$\PP^2\setminus D$, has the homotopy type of a finite CW complex of dimension $2$. The Euler characteristic of this cover
	is $pq\chi(\PP^2\setminus D)=pq(3-\chi(D))$. Thus, the higher homotopy groups of $\PP^2\setminus D$ are equal to
	those of the described cover, which by \cite[Proposition 3.3]{wall-FinitenessConditions} are determined by the rank of its fundamental group and its Euler characteristic.
\end{proof}

\begin{rem}
	By a theorem of MacLane and Whitehead \cite{MaWh}, the homotopy type of a finite $2$-dimensional CW complex is 
	completely determined by its algebraic $3$-type, which consists on its fundamental group $G$, its second homotopy 
	group $G_2$ (with the action of $G$), and its $k$-invariant, which is an element of $H^3(G,G_2)$ 
	(see also \cite[Chapter 2, Section 4]{bookCW}).
\end{rem}

%
%

\section{The homotopy type of fibered curve complements}\label{sec:homotopy}

Let $F:U\to \PP^1$ be a surjective algebraic morphism with connected generic fibers, where $U$ is the complement in $\PP^2$
of a finite number of points. In other words $F:U\to\PP^1$ is the restriction to its maximal domain of definition of a
primitive pencil $F:\PP^2\dashrightarrow \PP^1$. A pencil is called \emph{primitive} if its generic member is an irreducible
curve. In particular, primitive implies fixed component-free. Let $C=\overline{F^{-1}(B)}$ for some non-empty finite
set $B\subset\PP^1$. If $B_F\subset B$, we give a description as a $2$-dimensional CW complex of $\PP^2\setminus C$ with
the help of the locally trivial fibration $F:\PP^2\setminus C\to \PP^1\setminus B$ as follows. We begin by setting some notation:
\begin{enumerate}
	\item We denote the $r+1$ distinct points in $B$ by $B=\{P_0,P_1,\ldots, P_r\}$, with $r\geq 0$.
	\item $\PP^1\setminus B$ deformation retracts onto a wedge of circles, each of them corresponding to a closed loop $\gamma_k$ around $P_k$ (a meridian) for all $k=1,\ldots, r$. We denote by $E_B:= \bigvee_{k=1}^r \mathbb{S}^1$ this (strong) deformation retract of $\PP^1\setminus B$, and by $P\in \PP^1\setminus B_F$ the wedge point. We refer to $\gamma_k$ as an element of $\pi_1(E_B,P)$ or $\pi_1(\PP^1\setminus B,P)$. We will assume that $\gamma_k$ is positively oriented, where the positive orientation is induced by the complex structure on a big disk containing~$E_B$.
	\item $F^{-1}(P)$ is homeomorphic to a smooth genus $g$ curve with $s$ points removed. Hence, it deformation retracts onto a wedge of $2g+s-1$ circles. We denote by $A_P:= \bigvee_{i=1}^{2g+s-1} \mathbb{S}^1$ this (strong) deformation retract of $F^{-1}(P)$, and by $Q\in F^{-1}(P)$ the wedge point. We let $x_i$ be a loop around the $i$-th $\mathbb{S}^1$ component of $A_P$, and refer to it as an element of $\pi_1(A_P,Q)$ or $\pi_1(F^{-1}(P),Q)$.
	\item For all $k=1,\ldots, r$, let  $$M_k:F^{-1}(P)\to F^{-1}(P)$$ be a monodromy homeomorphism corresponding to the closed loop $\gamma_k$ (well defined up to isotopy).
	\item Let $R:F^{-1}(P)\to A_P$ be a retraction that makes $A_P$ into a strong deformation retract of $F^{-1}(P)$, and let $\iota:A_P\to F^{-1}(P)$ be the inclusion. Let $g_k:A_P\to A_P$ be a cellular map homotopic to $R\circ M_k\circ\iota$.
\end{enumerate}

The following result can also be obtained from Cohen-Suciu~\cite[\S~1.3]{CohenSuciu-iterated}, but we include a 
detailed proof in this note for the sake of completeness.

\begin{thm}\label{thm:CWfibration}
Let $F:\PP^2\dashrightarrow \PP^1$ be a primitive pencil. Let $B=\{P_0,\ldots,P_{r}\}\subset \PP^1$
be $r+1$ distinct points for $r\geq 0$ such that $B_F\subset B$. Let $\gamma_k, g_k$ and $x_i$ be as in the preceding discussion for all $k=1,\ldots,r$ and all $i=1,\ldots,2g+s-1$. 
%

Then, $\PP^2\setminus \overline{F^{-1}(B)}$ is homotopy equivalent to the aspherical 2-dimensional
CW complex with one $0$-cell given by the presentation
\begin{equation}\label{eq:pres}
\langle \gamma_k, x_i \mid  R_{k,i}(x)
\text{ for }k=1,\ldots, r,i=1,\ldots, 2g+s-1\rangle,
\end{equation}
where $R_{k,i}(x)=\gamma_k^{-1}x_i\gamma_k((g_k)_*(x_i))^{-1}$, and $(g_k)_*(x_i)$ is a word written in the  generators 
$$\{x_i\mid i=1,\ldots,2g+s-1\}$$
of $\pi_1(A_P,Q)$ for all $i=1,\ldots, 2g+s-1$.
\end{thm}
\begin{proof}
Let $Y$ be the $2$-dimensional CW complex with one $0$-cell corresponding to the presentation~\eqref{eq:pres}. Note that $Y$
can be obtained by gluing the mapping tori $T_{g_k}$ along $A_P\times\{0\}$ for all $k=1,\ldots, r$. Observe that,
since $F:\PP^2\setminus\overline{F^{-1}(B)}\to\PP^1\setminus B$ is a fibration. Consider $Y'=F^{-1}(E_B)$.
It suffices to show that $Y$ and $Y'$ are homotopy equivalent. Note that $Y$ are in the situation considered in
Lemma~\ref{lem:qf} and Corollary~\ref{cor:aspherical}, so it is aspherical and thus an Eilenberg-MacLane space.
Note that the long exact sequence associated with $F:Y\to E_B$ implies that $Y'$ is aspherical and its fundamental group
is given as a semidirect product $\pi_1(F^{-1}(P))\rtimes\pi_1(E_B)$ where the elements of $\pi_1(E_B)$ act
on $\pi_1(F^{-1}(P))$ by the geometric monodromy. Therefore $Y$ and $Y'$ have isomorphic fundamental groups.

\end{proof}

\begin{rem}
	For the complement of curves in affine space, Libgober gave a 2-dimensional CW complex description using braid monodromy arguments in \cite{Libgober-homotopytype}. His description holds for all affine curves, not just those that are the union of fibers of a certain map and contain all exceptional fibers. However, his arguments cannot be easily generalized to the projective case, so Theorem~\ref{thm:CWfibration} provides an extension to \cite{Libgober-homotopytype}.
\end{rem}

\begin{rem}
	Theorem~\ref{thm:CWfibration} is easily generalized (with the same proof) to quasi-projective surfaces of the form $U_B$ such that $B_F\subset B$ and $B\neq\emptyset$, where $F:U\to S$ is a surjective algebraic morphism from a smooth quasi-projective variety to a smooth projective curve $S$ with connected generic fibers. In this case, $S\setminus B$ also deformation retracts onto a wedge of circles, although not all of them will be meridians around points in $B$. However, the fact that those loops were meridians was not used in the proof at all.
\end{rem}

\begin{exam}
\label{example:cubics}
Theorem~\ref{thm:CWfibration} can be applied to a pencil of plane smooth cubics in $\PP^2$ intersecting the line at 
infinity at one point. Besides the line at infinity, which has multiplicity 3, such pencils have generically two nodal cubics 
as atypical fibers. Denote by $f:\CC^2\to \CC$ the affine map defining this pencil and denote by $B_f$ the set of atypical affine 
values for $f$. For instance, if $f(x,y)=x^3+x^2-y^2$, then $B_f=\{\frac{4}{27},0\}$. 
Choose $B=\{\lambda_1=\frac{4}{27},\lambda_2=0,\lambda_3,\ldots,\lambda_{r}\}\supset B_f$. 
According to Theorem~\ref{thm:CWfibration}, the homotopy type of $\CC^2\setminus f^{-1}(B)$ is given by the CW complex given 
by the presentation
\begin{equation}
\label{eq:presentation}
\left\langle 
\gamma_1,\gamma_2,...,\gamma_r,x_1,x_2: 
\array{ll}
{\gamma_1^{-1}x_1\gamma_1=x_1}, & {\gamma_1^{-1}x_2\gamma_1=x_1x_2},\\
{\gamma_2^{-1}x_1\gamma_2=x_2x_1}, & {\gamma_2^{-1}x_2\gamma_2=x_2},\\
{\gamma_3^{-1}x_1\gamma_3=x_1},&{\gamma_3^{-1}x_2\gamma_3=x_2},\\
...&\\
{\gamma_r^{-1}x_1\gamma_r=x_1},&{\gamma_r^{-1}x_2\gamma_r=x_2}.\\
\endarray
\right\rangle
\end{equation}
where $\gamma_k$ is the lift of a meridian around $\lambda_k$. 
The base point is shown as a black dot in Figure~\ref{fig:cubica}.
The loop $x_2$ is shown in red in Figure~\ref{fig:cubica}, whereas the loop $x_1$ travels the dotted part of $x_2$ in 
the direction of the arrow, then the blue oval in the counterclockwise direction and finally back along the dotted
arc in the direction opposite to the arrow. The action of the monodromy along $\gamma_1$ (resp. $\gamma_2$) produces a 
Dehn twist along the blue oval (resp. the red loop). This produces the relations
$\gamma_1^{-1}x_1\gamma_1=x_1$, $\gamma_1^{-1}x_2\gamma_1=x_1x_2$ and 
$\gamma_2^{-1}(x_1x_2)\gamma_2=x_2(x_1x_2)$ and $\gamma_2^{-1}x_2\gamma_2=x_2$, which imply 
$\gamma_2^{-1}x_1\gamma_2=x_2x_1$. The action of the monodromy along the remaining $\gamma_k$ is trivial.
This induces presentation~\eqref{eq:presentation} and by Theorem~\ref{thm:CWfibration} its associated 2-dimensional 
CW complex has the homotopy type of~$\CC^2\setminus f^{-1}(B)$.

Note that $\pi_1(\CC^2\setminus f^{-1}(B))$ is not free. However, quotienting $\pi_1(\CC^2\setminus f^{-1}(B))$ by the 
normal subgroup generated by either $\gamma_1$ or $\gamma_2$ yields 
$\pi_1\left(\CC^2\setminus f^{-1}(B\setminus\{\lambda_1\})\right)$ or 
$\pi_1\left(\CC^2\setminus f^{-1}(B\setminus\{\lambda_2\})\right)$, which is the free group $\FF_{r-1}$ in both cases.
\end{exam}


\begin{centering}
\begin{figure}
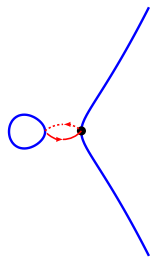

\scalebox{.3}{
\begingroup%
\makeatletter%
\begin{pgfpicture}%
\pgfpathrectangle{\pgfpointorigin}{\pgfqpoint{5.825000in}{4.700000in}}%
\pgfusepath{use as bounding box, clip}%
\begin{pgfscope}%
\pgfpathrectangle{\pgfqpoint{0.100000in}{0.100000in}}{\pgfqpoint{5.625000in}{4.500000in}}%
\pgfusepath{clip}%
\pgfsetbuttcap%
\pgfsetroundjoin%
\pgfsetlinewidth{3pt}%
\definecolor{currentstroke}{rgb}{0.000000,0.000000,1.000000}%
\pgfsetstrokecolor{currentstroke}%
\pgfsetdash{}{0pt}%
\pgfpathmoveto{\pgfqpoint{3.803828in}{4.513462in}}%
\pgfpathlineto{\pgfqpoint{3.801842in}{4.509436in}}%
\pgfpathlineto{\pgfqpoint{3.788974in}{4.484422in}}%
\pgfpathlineto{\pgfqpoint{3.774236in}{4.455382in}}%
\pgfpathlineto{\pgfqpoint{3.765543in}{4.438015in}}%
\pgfpathlineto{\pgfqpoint{3.759442in}{4.426342in}}%
\pgfpathlineto{\pgfqpoint{3.744475in}{4.397303in}}%
\pgfpathlineto{\pgfqpoint{3.729719in}{4.368263in}}%
\pgfpathlineto{\pgfqpoint{3.729243in}{4.367313in}}%
\pgfpathlineto{\pgfqpoint{3.714538in}{4.339223in}}%
\pgfpathlineto{\pgfqpoint{3.699555in}{4.310183in}}%
\pgfpathlineto{\pgfqpoint{3.692943in}{4.297176in}}%
\pgfpathlineto{\pgfqpoint{3.684416in}{4.281144in}}%
\pgfpathlineto{\pgfqpoint{3.669202in}{4.252104in}}%
\pgfpathlineto{\pgfqpoint{3.656644in}{4.227766in}}%
\pgfpathlineto{\pgfqpoint{3.654102in}{4.223064in}}%
\pgfpathlineto{\pgfqpoint{3.638649in}{4.194024in}}%
\pgfpathlineto{\pgfqpoint{3.623437in}{4.164985in}}%
\pgfpathlineto{\pgfqpoint{3.620344in}{4.158985in}}%
\pgfpathlineto{\pgfqpoint{3.607888in}{4.135945in}}%
\pgfpathlineto{\pgfqpoint{3.592443in}{4.106905in}}%
\pgfpathlineto{\pgfqpoint{3.584044in}{4.090852in}}%
\pgfpathlineto{\pgfqpoint{3.576911in}{4.077865in}}%
\pgfpathlineto{\pgfqpoint{3.561226in}{4.048826in}}%
\pgfpathlineto{\pgfqpoint{3.547745in}{4.023435in}}%
\pgfpathlineto{\pgfqpoint{3.545708in}{4.019786in}}%
\pgfpathlineto{\pgfqpoint{3.529778in}{3.990746in}}%
\pgfpathlineto{\pgfqpoint{3.514127in}{3.961706in}}%
\pgfpathlineto{\pgfqpoint{3.511445in}{3.956639in}}%
\pgfpathlineto{\pgfqpoint{3.498087in}{3.932666in}}%
\pgfpathlineto{\pgfqpoint{3.482200in}{3.903627in}}%
\pgfpathlineto{\pgfqpoint{3.475145in}{3.890488in}}%
\pgfpathlineto{\pgfqpoint{3.466144in}{3.874587in}}%
\pgfpathlineto{\pgfqpoint{3.450017in}{3.845547in}}%
\pgfpathlineto{\pgfqpoint{3.438846in}{3.825038in}}%
\pgfpathlineto{\pgfqpoint{3.433939in}{3.816507in}}%
\pgfpathlineto{\pgfqpoint{3.417566in}{3.787468in}}%
\pgfpathlineto{\pgfqpoint{3.402546in}{3.760283in}}%
\pgfpathlineto{\pgfqpoint{3.401462in}{3.758428in}}%
\pgfpathlineto{\pgfqpoint{3.384839in}{3.729388in}}%
\pgfpathlineto{\pgfqpoint{3.368563in}{3.700348in}}%
\pgfpathlineto{\pgfqpoint{3.366246in}{3.696124in}}%
\pgfpathlineto{\pgfqpoint{3.351826in}{3.671309in}}%
\pgfpathlineto{\pgfqpoint{3.335317in}{3.642269in}}%
\pgfpathlineto{\pgfqpoint{3.329946in}{3.632606in}}%
\pgfpathlineto{\pgfqpoint{3.318516in}{3.613229in}}%
\pgfpathlineto{\pgfqpoint{3.301774in}{3.584189in}}%
\pgfpathlineto{\pgfqpoint{3.293647in}{3.569756in}}%
\pgfpathlineto{\pgfqpoint{3.284900in}{3.555150in}}%
\pgfpathlineto{\pgfqpoint{3.267926in}{3.526110in}}%
\pgfpathlineto{\pgfqpoint{3.257347in}{3.507560in}}%
\pgfpathlineto{\pgfqpoint{3.250972in}{3.497070in}}%
\pgfpathlineto{\pgfqpoint{3.233765in}{3.468030in}}%
\pgfpathlineto{\pgfqpoint{3.221047in}{3.446001in}}%
\pgfpathlineto{\pgfqpoint{3.216724in}{3.438991in}}%
\pgfpathlineto{\pgfqpoint{3.199289in}{3.409951in}}%
\pgfpathlineto{\pgfqpoint{3.184748in}{3.385059in}}%
\pgfpathlineto{\pgfqpoint{3.182154in}{3.380911in}}%
\pgfpathlineto{\pgfqpoint{3.164495in}{3.351871in}}%
\pgfpathlineto{\pgfqpoint{3.148448in}{3.324707in}}%
\pgfpathlineto{\pgfqpoint{3.147259in}{3.322832in}}%
\pgfpathlineto{\pgfqpoint{3.129385in}{3.293792in}}%
\pgfpathlineto{\pgfqpoint{3.112148in}{3.264913in}}%
\pgfpathlineto{\pgfqpoint{3.112045in}{3.264752in}}%
\pgfpathlineto{\pgfqpoint{3.093967in}{3.235712in}}%
\pgfpathlineto{\pgfqpoint{3.076472in}{3.206673in}}%
\pgfpathlineto{\pgfqpoint{3.075849in}{3.205603in}}%
\pgfpathlineto{\pgfqpoint{3.058256in}{3.177633in}}%
\pgfpathlineto{\pgfqpoint{3.040620in}{3.148593in}}%
\pgfpathlineto{\pgfqpoint{3.039549in}{3.146764in}}%
\pgfpathlineto{\pgfqpoint{3.022279in}{3.119553in}}%
\pgfpathlineto{\pgfqpoint{3.004530in}{3.090514in}}%
\pgfpathlineto{\pgfqpoint{3.003249in}{3.088334in}}%
\pgfpathlineto{\pgfqpoint{2.986075in}{3.061474in}}%
\pgfpathlineto{\pgfqpoint{2.968251in}{3.032434in}}%
\pgfpathlineto{\pgfqpoint{2.966950in}{3.030222in}}%
\pgfpathlineto{\pgfqpoint{2.949708in}{3.003394in}}%
\pgfpathlineto{\pgfqpoint{2.931855in}{2.974355in}}%
\pgfpathlineto{\pgfqpoint{2.930650in}{2.972300in}}%
\pgfpathlineto{\pgfqpoint{2.913266in}{2.945315in}}%
\pgfpathlineto{\pgfqpoint{2.895450in}{2.916275in}}%
\pgfpathlineto{\pgfqpoint{2.894350in}{2.914389in}}%
\pgfpathlineto{\pgfqpoint{2.876884in}{2.887235in}}%
\pgfpathlineto{\pgfqpoint{2.859186in}{2.858196in}}%
\pgfpathlineto{\pgfqpoint{2.858050in}{2.856222in}}%
\pgfpathlineto{\pgfqpoint{2.840753in}{2.829156in}}%
\pgfpathlineto{\pgfqpoint{2.823286in}{2.800116in}}%
\pgfpathlineto{\pgfqpoint{2.821751in}{2.797394in}}%
\pgfpathlineto{\pgfqpoint{2.805158in}{2.771076in}}%
\pgfpathlineto{\pgfqpoint{2.788068in}{2.742037in}}%
\pgfpathlineto{\pgfqpoint{2.785451in}{2.737247in}}%
\pgfpathlineto{\pgfqpoint{2.770514in}{2.712997in}}%
\pgfpathlineto{\pgfqpoint{2.754003in}{2.683957in}}%
\pgfpathlineto{\pgfqpoint{2.749151in}{2.674648in}}%
\pgfpathlineto{\pgfqpoint{2.737449in}{2.654917in}}%
\pgfpathlineto{\pgfqpoint{2.721792in}{2.625878in}}%
\pgfpathlineto{\pgfqpoint{2.712852in}{2.607454in}}%
\pgfpathlineto{\pgfqpoint{2.706923in}{2.596838in}}%
\pgfpathlineto{\pgfqpoint{2.692506in}{2.567798in}}%
\pgfpathlineto{\pgfqpoint{2.679891in}{2.538758in}}%
\pgfpathlineto{\pgfqpoint{2.676552in}{2.529792in}}%
\pgfpathlineto{\pgfqpoint{2.667829in}{2.509719in}}%
\pgfpathlineto{\pgfqpoint{2.657313in}{2.480679in}}%
\pgfpathlineto{\pgfqpoint{2.648900in}{2.451639in}}%
\pgfpathlineto{\pgfqpoint{2.642590in}{2.422599in}}%
\pgfpathlineto{\pgfqpoint{2.640252in}{2.406462in}}%
\pgfpathlineto{\pgfqpoint{2.638033in}{2.393560in}}%
\pgfpathlineto{\pgfqpoint{2.635535in}{2.364520in}}%
\pgfpathlineto{\pgfqpoint{2.635535in}{2.335480in}}%
\pgfpathlineto{\pgfqpoint{2.638033in}{2.306440in}}%
\pgfpathlineto{\pgfqpoint{2.640252in}{2.293538in}}%
\pgfpathlineto{\pgfqpoint{2.642590in}{2.277401in}}%
\pgfpathlineto{\pgfqpoint{2.648900in}{2.248361in}}%
\pgfpathlineto{\pgfqpoint{2.657313in}{2.219321in}}%
\pgfpathlineto{\pgfqpoint{2.667829in}{2.190281in}}%
\pgfpathlineto{\pgfqpoint{2.676552in}{2.170208in}}%
\pgfpathlineto{\pgfqpoint{2.679891in}{2.161242in}}%
\pgfpathlineto{\pgfqpoint{2.692506in}{2.132202in}}%
\pgfpathlineto{\pgfqpoint{2.706923in}{2.103162in}}%
\pgfpathlineto{\pgfqpoint{2.712852in}{2.092546in}}%
\pgfpathlineto{\pgfqpoint{2.721792in}{2.074122in}}%
\pgfpathlineto{\pgfqpoint{2.737449in}{2.045083in}}%
\pgfpathlineto{\pgfqpoint{2.749151in}{2.025352in}}%
\pgfpathlineto{\pgfqpoint{2.754003in}{2.016043in}}%
\pgfpathlineto{\pgfqpoint{2.770514in}{1.987003in}}%
\pgfpathlineto{\pgfqpoint{2.785451in}{1.962753in}}%
\pgfpathlineto{\pgfqpoint{2.788068in}{1.957963in}}%
\pgfpathlineto{\pgfqpoint{2.805158in}{1.928924in}}%
\pgfpathlineto{\pgfqpoint{2.821751in}{1.902606in}}%
\pgfpathlineto{\pgfqpoint{2.823286in}{1.899884in}}%
\pgfpathlineto{\pgfqpoint{2.840753in}{1.870844in}}%
\pgfpathlineto{\pgfqpoint{2.858050in}{1.843778in}}%
\pgfpathlineto{\pgfqpoint{2.859186in}{1.841804in}}%
\pgfpathlineto{\pgfqpoint{2.876884in}{1.812765in}}%
\pgfpathlineto{\pgfqpoint{2.894350in}{1.785611in}}%
\pgfpathlineto{\pgfqpoint{2.895450in}{1.783725in}}%
\pgfpathlineto{\pgfqpoint{2.913266in}{1.754685in}}%
\pgfpathlineto{\pgfqpoint{2.930650in}{1.727700in}}%
\pgfpathlineto{\pgfqpoint{2.931855in}{1.725645in}}%
\pgfpathlineto{\pgfqpoint{2.949708in}{1.696606in}}%
\pgfpathlineto{\pgfqpoint{2.966950in}{1.669778in}}%
\pgfpathlineto{\pgfqpoint{2.968251in}{1.667566in}}%
\pgfpathlineto{\pgfqpoint{2.986075in}{1.638526in}}%
\pgfpathlineto{\pgfqpoint{3.003249in}{1.611666in}}%
\pgfpathlineto{\pgfqpoint{3.004530in}{1.609486in}}%
\pgfpathlineto{\pgfqpoint{3.022279in}{1.580447in}}%
\pgfpathlineto{\pgfqpoint{3.039549in}{1.553236in}}%
\pgfpathlineto{\pgfqpoint{3.040620in}{1.551407in}}%
\pgfpathlineto{\pgfqpoint{3.058256in}{1.522367in}}%
\pgfpathlineto{\pgfqpoint{3.075849in}{1.494397in}}%
\pgfpathlineto{\pgfqpoint{3.076472in}{1.493327in}}%
\pgfpathlineto{\pgfqpoint{3.093967in}{1.464288in}}%
\pgfpathlineto{\pgfqpoint{3.112045in}{1.435248in}}%
\pgfpathlineto{\pgfqpoint{3.112148in}{1.435087in}}%
\pgfpathlineto{\pgfqpoint{3.129385in}{1.406208in}}%
\pgfpathlineto{\pgfqpoint{3.147259in}{1.377168in}}%
\pgfpathlineto{\pgfqpoint{3.148448in}{1.375293in}}%
\pgfpathlineto{\pgfqpoint{3.164495in}{1.348129in}}%
\pgfpathlineto{\pgfqpoint{3.182154in}{1.319089in}}%
\pgfpathlineto{\pgfqpoint{3.184748in}{1.314941in}}%
\pgfpathlineto{\pgfqpoint{3.199289in}{1.290049in}}%
\pgfpathlineto{\pgfqpoint{3.216724in}{1.261009in}}%
\pgfpathlineto{\pgfqpoint{3.221047in}{1.253999in}}%
\pgfpathlineto{\pgfqpoint{3.233765in}{1.231970in}}%
\pgfpathlineto{\pgfqpoint{3.250972in}{1.202930in}}%
\pgfpathlineto{\pgfqpoint{3.257347in}{1.192440in}}%
\pgfpathlineto{\pgfqpoint{3.267926in}{1.173890in}}%
\pgfpathlineto{\pgfqpoint{3.284900in}{1.144850in}}%
\pgfpathlineto{\pgfqpoint{3.293647in}{1.130244in}}%
\pgfpathlineto{\pgfqpoint{3.301774in}{1.115811in}}%
\pgfpathlineto{\pgfqpoint{3.318516in}{1.086771in}}%
\pgfpathlineto{\pgfqpoint{3.329946in}{1.067394in}}%
\pgfpathlineto{\pgfqpoint{3.335317in}{1.057731in}}%
\pgfpathlineto{\pgfqpoint{3.351826in}{1.028691in}}%
\pgfpathlineto{\pgfqpoint{3.366246in}{1.003876in}}%
\pgfpathlineto{\pgfqpoint{3.368563in}{0.999652in}}%
\pgfpathlineto{\pgfqpoint{3.384839in}{0.970612in}}%
\pgfpathlineto{\pgfqpoint{3.401462in}{0.941572in}}%
\pgfpathlineto{\pgfqpoint{3.402546in}{0.939717in}}%
\pgfpathlineto{\pgfqpoint{3.417566in}{0.912532in}}%
\pgfpathlineto{\pgfqpoint{3.433939in}{0.883493in}}%
\pgfpathlineto{\pgfqpoint{3.438846in}{0.874962in}}%
\pgfpathlineto{\pgfqpoint{3.450017in}{0.854453in}}%
\pgfpathlineto{\pgfqpoint{3.466144in}{0.825413in}}%
\pgfpathlineto{\pgfqpoint{3.475145in}{0.809512in}}%
\pgfpathlineto{\pgfqpoint{3.482200in}{0.796373in}}%
\pgfpathlineto{\pgfqpoint{3.498087in}{0.767334in}}%
\pgfpathlineto{\pgfqpoint{3.511445in}{0.743361in}}%
\pgfpathlineto{\pgfqpoint{3.514127in}{0.738294in}}%
\pgfpathlineto{\pgfqpoint{3.529778in}{0.709254in}}%
\pgfpathlineto{\pgfqpoint{3.545708in}{0.680214in}}%
\pgfpathlineto{\pgfqpoint{3.547745in}{0.676565in}}%
\pgfpathlineto{\pgfqpoint{3.561226in}{0.651174in}}%
\pgfpathlineto{\pgfqpoint{3.576911in}{0.622135in}}%
\pgfpathlineto{\pgfqpoint{3.584044in}{0.609148in}}%
\pgfpathlineto{\pgfqpoint{3.592443in}{0.593095in}}%
\pgfpathlineto{\pgfqpoint{3.607888in}{0.564055in}}%
\pgfpathlineto{\pgfqpoint{3.620344in}{0.541015in}}%
\pgfpathlineto{\pgfqpoint{3.623437in}{0.535015in}}%
\pgfpathlineto{\pgfqpoint{3.638649in}{0.505976in}}%
\pgfpathlineto{\pgfqpoint{3.654102in}{0.476936in}}%
\pgfpathlineto{\pgfqpoint{3.656644in}{0.472234in}}%
\pgfpathlineto{\pgfqpoint{3.669202in}{0.447896in}}%
\pgfpathlineto{\pgfqpoint{3.684416in}{0.418856in}}%
\pgfpathlineto{\pgfqpoint{3.692943in}{0.402824in}}%
\pgfpathlineto{\pgfqpoint{3.699555in}{0.389817in}}%
\pgfpathlineto{\pgfqpoint{3.714538in}{0.360777in}}%
\pgfpathlineto{\pgfqpoint{3.729243in}{0.332687in}}%
\pgfpathlineto{\pgfqpoint{3.729719in}{0.331737in}}%
\pgfpathlineto{\pgfqpoint{3.744475in}{0.302697in}}%
\pgfpathlineto{\pgfqpoint{3.759442in}{0.273658in}}%
\pgfpathlineto{\pgfqpoint{3.765543in}{0.261985in}}%
\pgfpathlineto{\pgfqpoint{3.774236in}{0.244618in}}%
\pgfpathlineto{\pgfqpoint{3.788974in}{0.215578in}}%
\pgfpathlineto{\pgfqpoint{3.801842in}{0.190564in}}%
\pgfpathlineto{\pgfqpoint{3.803828in}{0.186538in}}%
\pgfpathmoveto{\pgfqpoint{1.587561in}{2.063138in}}%
\pgfpathlineto{\pgfqpoint{1.623861in}{2.057176in}}%
\pgfpathlineto{\pgfqpoint{1.660161in}{2.056085in}}%
\pgfpathlineto{\pgfqpoint{1.696460in}{2.059411in}}%
\pgfpathlineto{\pgfqpoint{1.732760in}{2.066695in}}%
\pgfpathlineto{\pgfqpoint{1.757755in}{2.074122in}}%
\pgfpathlineto{\pgfqpoint{1.769060in}{2.077855in}}%
\pgfpathlineto{\pgfqpoint{1.805359in}{2.093223in}}%
\pgfpathlineto{\pgfqpoint{1.825134in}{2.103162in}}%
\pgfpathlineto{\pgfqpoint{1.841659in}{2.112506in}}%
\pgfpathlineto{\pgfqpoint{1.872490in}{2.132202in}}%
\pgfpathlineto{\pgfqpoint{1.877959in}{2.136195in}}%
\pgfpathlineto{\pgfqpoint{1.909424in}{2.161242in}}%
\pgfpathlineto{\pgfqpoint{1.914259in}{2.165732in}}%
\pgfpathlineto{\pgfqpoint{1.939192in}{2.190281in}}%
\pgfpathlineto{\pgfqpoint{1.950558in}{2.203711in}}%
\pgfpathlineto{\pgfqpoint{1.963317in}{2.219321in}}%
\pgfpathlineto{\pgfqpoint{1.982305in}{2.248361in}}%
\pgfpathlineto{\pgfqpoint{1.986858in}{2.257645in}}%
\pgfpathlineto{\pgfqpoint{1.996415in}{2.277401in}}%
\pgfpathlineto{\pgfqpoint{2.005780in}{2.306440in}}%
\pgfpathlineto{\pgfqpoint{2.010463in}{2.335480in}}%
\pgfpathlineto{\pgfqpoint{2.010463in}{2.364520in}}%
\pgfpathlineto{\pgfqpoint{2.005780in}{2.393560in}}%
\pgfpathlineto{\pgfqpoint{1.996415in}{2.422599in}}%
\pgfpathlineto{\pgfqpoint{1.986858in}{2.442355in}}%
\pgfpathlineto{\pgfqpoint{1.982305in}{2.451639in}}%
\pgfpathlineto{\pgfqpoint{1.963317in}{2.480679in}}%
\pgfpathlineto{\pgfqpoint{1.950558in}{2.496289in}}%
\pgfpathlineto{\pgfqpoint{1.939192in}{2.509719in}}%
\pgfpathlineto{\pgfqpoint{1.914259in}{2.534268in}}%
\pgfpathlineto{\pgfqpoint{1.909424in}{2.538758in}}%
\pgfpathlineto{\pgfqpoint{1.877959in}{2.563805in}}%
\pgfpathlineto{\pgfqpoint{1.872490in}{2.567798in}}%
\pgfpathlineto{\pgfqpoint{1.841659in}{2.587494in}}%
\pgfpathlineto{\pgfqpoint{1.825134in}{2.596838in}}%
\pgfpathlineto{\pgfqpoint{1.805359in}{2.606777in}}%
\pgfpathlineto{\pgfqpoint{1.769060in}{2.622145in}}%
\pgfpathlineto{\pgfqpoint{1.757755in}{2.625878in}}%
\pgfpathlineto{\pgfqpoint{1.732760in}{2.633305in}}%
\pgfpathlineto{\pgfqpoint{1.696460in}{2.640589in}}%
\pgfpathlineto{\pgfqpoint{1.660161in}{2.643915in}}%
\pgfpathlineto{\pgfqpoint{1.623861in}{2.642824in}}%
\pgfpathlineto{\pgfqpoint{1.587561in}{2.636862in}}%
\pgfpathlineto{\pgfqpoint{1.552251in}{2.625878in}}%
\pgfpathlineto{\pgfqpoint{1.551262in}{2.625536in}}%
\pgfpathlineto{\pgfqpoint{1.514962in}{2.606560in}}%
\pgfpathlineto{\pgfqpoint{1.501342in}{2.596838in}}%
\pgfpathlineto{\pgfqpoint{1.478662in}{2.578625in}}%
\pgfpathlineto{\pgfqpoint{1.468189in}{2.567798in}}%
\pgfpathlineto{\pgfqpoint{1.443610in}{2.538758in}}%
\pgfpathlineto{\pgfqpoint{1.442363in}{2.537040in}}%
\pgfpathlineto{\pgfqpoint{1.426357in}{2.509719in}}%
\pgfpathlineto{\pgfqpoint{1.412179in}{2.480679in}}%
\pgfpathlineto{\pgfqpoint{1.406063in}{2.465018in}}%
\pgfpathlineto{\pgfqpoint{1.401726in}{2.451639in}}%
\pgfpathlineto{\pgfqpoint{1.394665in}{2.422599in}}%
\pgfpathlineto{\pgfqpoint{1.389958in}{2.393560in}}%
\pgfpathlineto{\pgfqpoint{1.387605in}{2.364520in}}%
\pgfpathlineto{\pgfqpoint{1.387605in}{2.335480in}}%
\pgfpathlineto{\pgfqpoint{1.389958in}{2.306440in}}%
\pgfpathlineto{\pgfqpoint{1.394665in}{2.277401in}}%
\pgfpathlineto{\pgfqpoint{1.401726in}{2.248361in}}%
\pgfpathlineto{\pgfqpoint{1.406063in}{2.234982in}}%
\pgfpathlineto{\pgfqpoint{1.412179in}{2.219321in}}%
\pgfpathlineto{\pgfqpoint{1.426357in}{2.190281in}}%
\pgfpathlineto{\pgfqpoint{1.442363in}{2.162960in}}%
\pgfpathlineto{\pgfqpoint{1.443610in}{2.161242in}}%
\pgfpathlineto{\pgfqpoint{1.468189in}{2.132202in}}%
\pgfpathlineto{\pgfqpoint{1.478662in}{2.121375in}}%
\pgfpathlineto{\pgfqpoint{1.501342in}{2.103162in}}%
\pgfpathlineto{\pgfqpoint{1.514962in}{2.093440in}}%
\pgfpathlineto{\pgfqpoint{1.551262in}{2.074464in}}%
\pgfpathlineto{\pgfqpoint{1.552251in}{2.074122in}}%
\pgfpathlineto{\pgfqpoint{1.587561in}{2.063138in}}%
\pgfpathclose%
\pgfusepath{stroke}%
\end{pgfscope}%
\begin{pgfscope}

\end{pgfscope}
\pgfpathrectangle{\pgfqpoint{0.100000in}{0.100000in}}{\pgfqpoint{5.625000in}{4.500000in}}%
\pgfusepath{clip}%
\pgfsetbuttcap%
\pgfsetroundjoin%
\pgfsetlinewidth{2pt}%
\pgfsetdash{}{0cm}
\definecolor{currentstroke}{rgb}{0.000000,1.000000,0.000000}%
\pgfsetstrokecolor{currentstroke}%
\pgfpathcircle{\pgfpoint{6.7cm}{6cm}}{2mm}
  \pgfusepath{fill}
  \pgfusepath{stroke}%
\pgfsetdash{{0.1cm}{0.1cm}{0.1cm}{0.1cm}}{0cm}
\definecolor{currentstroke}{rgb}{1.000000,0.000000,0.000000}%
\pgfsetstrokecolor{currentstroke}%
\pgfpathmoveto{\pgfpoint{6.7cm}{6cm}}
  \pgfsetarrows{-latex}
  \pgfsetshortenstart{4pt}
\pgfpatharc{50}{90}{70cm/2cm}
  \pgfusepath{draw}
\pgfusepath{stroke}%
\pgfsetarrows{-}
\pgfsetarrowsstart{}
\pgfsetshortenstart{0pt}
\pgfsetdash{{0.1cm}{0.1cm}{0.1cm}{0.1cm}}{0cm}
\definecolor{currentstroke}{rgb}{1.000000,0.000000,0.000000}%
\pgfpathmoveto{\pgfpoint{5.1cm}{6cm}}
\pgfpatharc{130}{90}{70cm/2cm}
  \pgfusepath{draw}
\pgfusepath{stroke}%
\pgfsetdash{}{0cm}
\definecolor{currentstroke}{rgb}{1.000000,0.000000,0.000000}%
\pgfpathmoveto{\pgfpoint{6.7cm}{6cm}}
\pgfpatharc{140}{90}{-60cm/2cm}
  \pgfusepath{draw}
\pgfusepath{stroke}%
  \pgfsetarrows{-latex}
  \pgfsetshortenstart{4pt}
\pgfpathmoveto{\pgfpoint{5.1cm}{6cm}}
\pgfpatharc{40}{90}{-60cm/2cm}
  \pgfusepath{draw}
\pgfusepath{stroke}%
\end{pgfpicture}%
\makeatother%
\endgroup%
}
\caption{Real picture of $f^{-1}(\frac{2}{27})$}
\label{fig:cubica}
\end{figure}
\end{centering}

\bibliographystyle{amsplain}
\providecommand{\bysame}{\leavevmode\hbox to3em{\hrulefill}\thinspace}
\providecommand{\MR}{\relax\ifhmode\unskip\space\fi MR }
\providecommand{\MRhref}[2]{%
  \href{http://www.ams.org/mathscinet-getitem?mr=#1}{#2}
}
\providecommand{\href}[2]{#2}

\end{document}